\documentclass[
          submission,copyright,creativecommons]{eptcs}
\providecommand{\event}{QPL 2016}
\usepackage{breakurl}
\usepackage[warn]{textcomp}
\usepackage[english]{babel}
\usepackage{amssymb}
\usepackage{amsmath}
\usepackage{amsthm}
\usepackage{mathtools}
\usepackage{nicefrac}
\usepackage{mathrsfs}
\usepackage{stmaryrd}
\input xy
\xyoption{all}
\usepackage{enumitem}
\usepackage{multicol}
\usepackage{tikz}

\newcommand{\shortlong}[2]{#2}

\newcommand{\R}{\mathbb{R}}

\newcommand{\N}{\mathbb{N}}

\newcommand{\ra}{\rightarrow}

\newcommand{\eqdf}{:=}

\newcommand{\CP}{\mathbf{C}^*_\mathrm{P}}
\newcommand{\CPU}{\mathbf{C}^*_\mathrm{PU}}
\newcommand{\CCPU}{\mathbf{CC}^*_\mathrm{PU}}
\newcommand{\WNPU}{\mathbf{W}^*_\mathrm{NPU}}
\newcommand{\CPsU}{\mathbf{C}^*_\mathrm{PsU}}
\newcommand{\CcPU}{\mathbf{C}^*_\mathrm{CPU}}
\newcommand{\C}{\mathbb{C}}
\newcommand{\CMIU}{\mathbf{C}^*_\mathrm{MIU}}
\newcommand{\CCMIU}{\mathbf{CC}^*_\mathrm{MIU}}
\newcommand{\WNMIU}{\mathbf{W}^*_\mathrm{NMIU}}
\newcommand{\Kl}[1]{\mathcal{K}\ell(#1)}
\newcommand{\EM}[1]{\mathcal{EM}(#1)}
\newcommand{\Cat}[1]{\mathbf{#1}}
\newcommand{\id}[1]{\mathrm{id}_{#1}}
\newcommand{\Ran}[1]{\mathrm{Ran}#1}
\newcommand{\afterKl}{\mathrel{\varodot}}

\newcommand{\Mtwo}{\mathrm{M}_2}

\setlist[1]{label=(\roman*)}

\theoremstyle{definition}
\newtheorem{numbering}{Use to get sequential numbering}
\newtheorem{dfn}[numbering]{Definition}

\newtheorem{ex}[numbering]{Example}

\theoremstyle{remark}
\newtheorem{rem}[numbering]{Remark}

\theoremstyle{notation}
\newtheorem{nt}[numbering]{Notation}

\theoremstyle{theorem}
\newtheorem{lem}[numbering]{Lemma}
\newtheorem{cor}[numbering]{Corollary}
\newtheorem{thm}[numbering]{Theorem}
\newtheorem{prob}[numbering]{Problem}

\title{Quantum Programs as Kleisli Maps}
\author{Abraham Westerbaan
\institute{Radboud University Nijmegen}
\email{bram@westerbaan.name}}

\def\titlerunning{Quantum Programs as Kleisli Maps}
\def\authorrunning{A.A.~Westerbaan}
\begin{document}

{ 
\maketitle
\begin{abstract}
        Furber and Jacobs
        have shown in their study of quantum computation
        that the category of commutative $C^*$-algebras
        and \emph{PU-maps} (positive linear maps which preserve the unit)
        is isomorphic to the Kleisli category
        of a comonad on the category of
        commutative $C^*$-algebras with
        \emph{MIU-maps} (linear maps which preserve multiplication,
        involution and unit).~\cite{furber2013}

        In this paper,
        we prove a non-commutative variant of this result:
        the category of $C^*$-algebras and PU-maps
        is isomorphic to the Kleisli category
        of a comonad on the subcategory
        of MIU-maps.

	A variation on this result
	has been used
	to construct a model of Selinger and Valiron's
	quantum lambda calculus using von Neumann 
	algebras.~\cite{CW2016}
\end{abstract}
 }

{         The semantics of a 
        non-deterministic program that takes two bits
        and returns three bits can
        be described as a multimap (= binary relation) from~$\{0,1\}^2$
        to~$\{0,1\}^3$.
        Similarly,
        a program
        that takes two qubits
        and returns three qubits
        can be modelled as a positive
        linear unit-preserving map
        from $\Mtwo\,\otimes\,\Mtwo\,\otimes\,\Mtwo$
        to~$\Mtwo\,\otimes\,\Mtwo$,
        where~$\Mtwo$ 
        is the $C^*$-algebra
        of  $2\times 2$-matrices over~$\C$.
        
        More generally,
        the category~$\Cat{Set}_\mathrm{multi}$
        of multimaps between sets
        models non-deterministic programs (running
        on an ordinary computer),
        while the opposite of the category~$\CPU$
        of \emph{PU-maps} (positive linear unit-preserving maps)
        between $C^*$-algebras 
        models programs running on a quantum computer.
        (When we write ``$C^*$-algebra''
        we always mean ``$C^*$-algebra with unit''.)

        A multimap from~$\{0,1\}^2$ to~$\{0,1\}^3$
        is simply a map from~$\{0,1\}^2$ to~$\mathcal{P}(\{0,1\}^3)$.
        In the same line~$\Cat{Set}_{\mathrm{multi}}$
        is (isomorphic to) the Kleisli category
        of the powerset monad~$\mathcal{P}$ on~$\Cat{Set}$.
        What about $\CPU$?

        We will show that there is a monad~$\Omega$
        on~$(\CMIU)^\mathrm{op}$, the opposite of the 
        category~$\CMIU$ of $C^*$-algebras and \emph{MIU-maps}
        (linear maps 
        that preserve the 
        multiplication, involution and unit),
        such that~$(\CPU)^\mathrm{op}$
        is isomorphic to the Kleisli category of~$\Omega$.
        We say that $(\CPU)^\mathrm{op}$ is \emph{Kleislian} 
        over~$(\CMIU)^\mathrm{op}$.
        So in the same way we add non-determinism 
        to~$\Cat{Set}$ by the powerset monad~$\mathcal{P}$
        yielding~$\Cat{Set}_\mathrm{multi}$,
        we can obtain $(\CPU)^\mathrm{op}$
        from~$(\CMIU)^\mathrm{op}$
        by a monad~$\Omega$.

        Let us spend some words on how we obtain
        this monad~$\Omega$.
        Note that since every positive element of a $C^*$-algebra~$\mathscr{A}$
        is of the form~$a^*a$ for some~$a\in\mathscr{A}$
        any MIU-map will be positive.
        Thus $\CMIU$ is a subcategory of~$\CPU$.
        Let $U\colon \CMIU\longrightarrow \CPU$ be the embedding.

        In Section~\ref{S:left-adjoint}
        we will prove that~$U$ has a left adjoint
        $F\colon \CPU\longrightarrow \CMIU$,
        see Theorem~\ref{thm:left-adjoint}.
        This adjunction gives us a comonad~$\Omega\eqdf FU$ on~$\CMIU$
        (which is a monad on~$(\CMIU)^\mathrm{op}$)
        with the same counit as the adjunction.
        The comultiplication~$\delta$ is given by 
        $\delta_{\mathscr{A}}= F\eta_{U\mathscr{A}}$ for
        every object~$\mathscr{A}$ from~$\CMIU$
        where~$\eta$ is the unit of the adjunction
        between~$F$ and~$U$.

        In Section~\ref{S:kleisli} we will prove that~$(\CPU)^\mathrm{op}$
        is isomorphic to~$\Kl{FU}$ if~$FU$ is considered a monad 
        on~$(\CMIU)^\mathrm{op}$.
        In fact, we will prove that
        the \emph{comparison functor}
        $L\colon \Kl{FU}\longrightarrow (\CPU)^\mathrm{op}$
        (which sends a MIU-map 
        $f\colon FU\mathscr{A}\longrightarrow \mathscr{B}$
        to 
        $Uf\circ \eta_{U\mathscr{A}}\colon 
        U\mathscr{A}\longrightarrow U\mathscr{B}$)
        is an isomorphism,
        see Corollary~\ref{C:kleisli}.

        The method used to show that~$(\CPU)^\mathrm{op}$
        is Kleislian over~$(\CMIU)^\mathrm{op}$ is quite general
        and it will be obvious that many
        variations
        on~$(\CPU)^\mathrm{op}$ will be Kleislian over~$(\CMIU)^\mathrm{op}$
        as well, such as the opposite of the category
        of subunital completely positive linear maps between $C^*$-algebras.
        The flip-side of this generality is that we discover preciously little 
        about the monad~$\Omega$
        which leaves room for future inquiry
        (see Section~\ref{S:discussion}).

	We will also see
	that the opposite~$(\mathbf{W}^*_\mathrm{NCPsU})^\mathrm{op}$ of 
	the category of
	normal completely positive
	subunital maps between von Neumann algebras
	is Kleislian over the 
	subcategory~$(\mathbf{W}^*_\mathrm{NMIU})^\mathrm{op}$
	of normal unital $*$-homomorphisms.
	This fact is used in~\cite{CW2016}
	to construct 
	an adequate model of 
	Selinger and Valiron's quantum lambda calculus
	using von Neumann algebras.

\section{The Left Adjoint}
\label{S:left-adjoint}

        In Theorem~\ref{thm:left-adjoint} we will show 
        that~$U$ has a left adjoint, $F\colon \CMIU\ra \CPU$,
        using a quite general method.
        As a result we do not get any ``concrete'' information
        about~$F$ in the sense that while we will learn that
        for every $C^*$-algebra $\mathscr{A}$
        there exists an arrow 
        $\varrho \colon\mathscr{A}\ra UF\mathscr{A}$
        which is initial from~$\mathscr{A}$ to~$U$
        we will learn nothing more about~$\varrho$ than this.
        Nevertheless, for some (very) basic $C^*$-algebras~$\mathscr{A}$
        we can describe~$F\mathscr{A}$ directly,
        as is shown below in Example~\ref{ex:C}--\ref{ex:C3}.
\begin{ex}
\label{ex:C}
        Let us start easy: $\C$ will be mapped to itself by~$F$, that is:\\
                \emph{the identity  $\varrho \colon \C \longrightarrow U\C$
                is an initial arrow from~$\C$ to~$U(-)$.}\\
        Indeed, 
        let $\mathscr{A}$ be a $C^*$-algebra
        and let $\sigma \colon \C\ra U\mathscr{A}$ be a PU-map.
        Then~$\sigma$ must be given by $\sigma(\lambda) = \lambda\cdot 1$
        for~$\lambda\in\C$, where~$1$ is the identity of~$\mathscr{A}$.
        Thus~$\sigma$ is a MIU-map as well.  Hence there is a unique
        MIU-map $\hat{\sigma}\colon \C \ra \mathscr{A}$ 
        (namely $\hat{\sigma}=\sigma$)
        such that $\hat{\sigma}\circ \varrho = \sigma$.
        ($\C$ is initial in both~$\CMIU$ and~$\CPU$.)
\end{ex}
\begin{ex}
        \label{ex:C2}
        The image of~$\C^2$ under~$F$ will be
        the $C^*$-algebra~$C[0,1]$ of continuous functions
        from~$[0,1]$ to~$\C$.
        As will become clear below,
        this is very much related to the familiar functional calculus
        for $C^*$-algebras:
        given an element~$a$ of a $C^*$-algebra~$\mathscr{A}$
        with $0\leq a\leq 1$
        and $f\in C[0,1]$
        we can make sense of~``$f(a)$'',
        as an element of~$\mathscr{A}$.\\
                \emph{The map $\varrho\colon \C^2 \longrightarrow UC[0,1]$
        given by, for $\lambda,\mu\in\C$, $x\in [0,1]$,
        \begin{equation*}
                \varrho(\lambda,\mu)(x)
                \ = \ 
                \lambda x \,+\, \mu(1-x)
        \end{equation*}
        is an initial arrow from~$\C^2$ to~$U$.}\\
        Let~$\sigma\colon \C^2 \ra U\mathscr{A}$ be a PU-map.
        We must show that there is a unique MIU-map 
        $\overline{\sigma}\colon C[0,1]\ra \mathscr{A}$
        such that $\sigma = \overline{\sigma}\circ\varrho$.

        Writing $a\eqdf \sigma(1,0)$, we have
                $\sigma(\lambda,\mu) =  \lambda a + \mu(1-a)$
        for all $\lambda,\mu\in \C$.
        Note that~$(0,0)\leq (1,0)\leq (1,1)$ and
        thus $0\leq a\leq 1$.
        Let $C^*(a)$ be the $C^*$-subalgebra of~$\mathscr{A}$
        generated by~$a$.
        Then $C^*(a)$ is commutative since~$a$ is positive (and thus normal).
        Given a MIU-map $\omega\colon C^*(a)\ra \C$
        we have $\omega(a)\in[0,1]$ since $0\leq a\leq 1$.
        Thus $\omega\mapsto \omega(a)$
        gives a map $j\colon \Sigma C^*(a)\ra [0,1]$,
        where $\Sigma C^*(a)$ is the spectrum of~$C^*(a)$,
        that is, $\Sigma C^*(a)$ is
        the set of MIU-maps from~$C^*(a)$ to~$\C$ with
        the topology of pointwise convergence.
        (By the way, the image of~$j$ is the spectrum of
        the \emph{element}~$a$.)
        The map~$j$ is continuous since the topology on~$\Sigma C^*(a)$ is 
        induced by the product topology on~$\C^{C^*(a)}$.
        Thus the assignment $h\mapsto h\circ j$
        gives a MIU-map $Cj\colon C[0,1]\rightarrow C\Sigma C^*(a)$.
        By Gelfand's representation theorem
        there is a MIU-isomorphism 
        \begin{equation*}
                \gamma\colon C^*(a)\longrightarrow C\Sigma C^*(a)
        \end{equation*}
        given by $\gamma(b)(\omega)=\omega(b)$ for all~$b\in C^*(a)$ 
        and~$\omega\in \Sigma C^*(a)$.
        Now, define
        \begin{equation*}
                \overline{\sigma}\eqdf \gamma^{-1} \circ C j\colon  \ 
                C[0,1] \longrightarrow \C^*(a) \hookrightarrow \mathscr{A}.
        \end{equation*}
        (In the language of the functional calculus,
        $\overline{\sigma}$ maps $f$ to~$f(a)$.)
        We claim that~$\overline{\sigma}\circ \varrho = \sigma$.
        It suffices to show that~$Cj\circ \varrho \equiv
        \gamma\circ \overline{\sigma}\circ \varrho
        =\gamma\circ\sigma$.
        Let $\lambda,\mu\in \C$ and $\omega\in \Sigma C^*(a)$ be given. 
        We have
        \begin{alignat*}{3}
                (Cj\circ\varrho)(\lambda,\mu)(\omega) 
                \ &=\ (Cj)(\varrho(\lambda,\mu))(\omega)  \\
                \ &=\ \varrho(\lambda,\mu)(\,j(\omega)\,) 
                &&\text{by def.~of~$Cj$} \\
                \ &=\ \lambda j(\omega) \,+\, \mu(1-j(\omega)) 
                &&\text{by def.~of~$\varrho$} \\
                \ &=\ \lambda \omega(a) \,+\, \mu(1-\omega(a)) 
                \qquad&&\text{by def.~of~$j$} \\
                \ &=\ \omega(\,\lambda a + \mu(1-a)\,) 
                &&\text{as $\omega$~is a MIU-map} \\
                \ &=\ \omega(\,\sigma(\lambda,\mu)\,) 
                &&\text{by choice of~$a$} \\
                \ &=\ \gamma( \,\sigma(\lambda,\mu)\, )(\omega).
                &&\text{by def.~of~$\gamma$} \\
                \ &=\ (\gamma\circ\sigma)(\lambda,\mu)(\omega).
        \end{alignat*}
        It remains to be shown that~$\overline{\sigma}$
        is the only MIU-map $\tau\colon C[0,1]\ra \mathscr{A}$
        such that $U\tau \circ \varrho = \sigma$.
        Let~$\tau$ be such a map; we prove that~$\tau=\overline{\sigma}$.
        By assumption $\tau$ and~$\overline{\sigma}$ agree
        on the elements $f\in C[0,1]$ of the form
        \begin{equation*}
                f(x)\ = \ \lambda x \,+\, \mu (1-x).
        \end{equation*}
        In particular, $\overline{\sigma}$ and~$\tau$
        agree on the map $h\colon[0,1]\ra \C$ given by $h(x)=x$.
        
        Now, since
         $\overline{\sigma}$ and~$\tau$ are MIU-maps
         and $h$ generates the $C^*$-algebra~$C[0,1]$ 
        (this is Weierstrass's theorem), 
        it follows that~$\overline{\sigma} = \tau$.
\end{ex}
\begin{ex}
        \label{ex:C3}
        The image of~$\C^3$ under~$F$ will not be commutative,
        or more formally:\\
                \emph{If $\varrho\colon \C^3\longrightarrow U\mathscr{B}$
        is an initial map from~$\C^3$ to~$U$,
        then~$\mathscr{B}$ is not commutative.}\\
        Suppose that~$\mathscr{B}$ is commutative towards contradiction.
        Let~$\mathscr{A}$ be a $C^*$-algebra
        in which there are positive $a_1$, $a_2$, $a_3$
        such that $a_1a_2 \neq a_2 a_1$ and $a_1 + a_2 + a_3 = 1$.

        (For example,
        we can take~$\mathscr{A}$ to be the set of linear operators on~$\C^2$
        and let
        \begin{align*}
                a_1 \ &\eqdf \ \nicefrac{1}{2} \,P_1 &
                a_2 \ &\eqdf \ \nicefrac{1}{2} \,P_+ &
                a_3 \ &\eqdf\  I \,-\, \nicefrac{1}{2} \,P_1 \,-\, 
                           \nicefrac{1}{2}\, P_+
        \end{align*}
        where $P_1$ denotes the orthogonal projection 
        onto $\{\,(0,x)\colon x\in \C\,\}$
        and $P_+$ is the orthogonal projection 
        onto $\{\,(x,x)\colon x\in\C\,\}$.)

        Define $f\colon \C^3 \ra \mathscr{A}$ by,
        for all~$\lambda_1,\lambda_2,\lambda_3\in \C$,
        \begin{equation*}
                f(\lambda_1,\lambda_2,\lambda_3)
                \ = \ 
                \lambda_1a_1 \,+\,
                \lambda_2a_2 \,+\,
                \lambda_3a_3.
        \end{equation*}
        Then it is not hard to see that~$f$ a PU-map.
        So as~$\mathscr{B}$ is the initial 
        arrow from~$\C^3$ to~$U$ there is a (unique) MIU-map 
        $\overline{f}\colon \mathscr{B}\ra \mathscr{A}$
        such that $\overline{f}\circ \varrho = f$.
        We have
        \begin{alignat*}{3}
                a_1\cdot a_2 
                \ &=\ f(1,0,0) \cdot f(0,1,0) \\
                 &=\ \overline{f}(\varrho(1,0,0))
                \,\cdot\, \overline{f}(\varrho(0,1,0)) \\
                 &=\ \overline{f}(\, \varrho(1,0,0)\,\cdot\,
                        \varrho(0,1,0)\,) \\
                 &=\ \overline{f}(\, \varrho(0,1,0)\,\cdot\,
                        \varrho(1,0,0)\,)
                        \qquad&&\text{because $\mathscr{B}$ is commutative}\\
                              &=\ \overline{f}( \varrho(0,1,0))\,\cdot\,
                        \overline{f}(\varrho(1,0,0))\\
                &=\ a_2 \cdot a_1.
        \end{alignat*}
        This contradicts $a_1 \cdot a_2 \neq  a_2 \cdot a_1$.
        Hence~$\mathscr{B}$ is not commutative.
\end{ex}
\begin{rem}
        Before we prove
        that the embedding
        $\CMIU\ra \CPU$
        has a left adjoint~$F$
        (see Theorem~\ref{thm:left-adjoint})
        let us compare
        what we already know about~$F$ with the commutative case.
        Let~$\CCMIU$ denote
        the category of MIU-maps between commutative $C^*$-algebras
        and let~$\CCPU$
        denote the category of PU-maps between commutative
        $C^*$-algebras.
        From the work in~\cite{furber2013}
        it follows that the embedding~$\CCMIU\longrightarrow \CCPU$
        has a left adjoint~$F'$
        and moreover
        that $F'\mathscr{A} = C \mathrm{Stat} \mathscr{A}$,
        where $\mathrm{Stat}\mathscr{A}$
        is the topological
        space of PU-maps from~$\mathscr{A}$ to~$\C$
        with pointwise convergence
        and $C\mathrm{Stat}\mathscr{A}$
        is the $C^*$-algebra of continuous
        functions from~$\mathrm{Stat}\mathscr{A}$
        to~$\C$.

        Let $x\in [0,1]$.
        Then the assignment $(\lambda,\mu)\mapsto x\lambda + (1-x)\mu$
        gives a PU-map $\overline{x}\colon \C^2 \ra \C$.
        It is not hard to see that $x\mapsto \overline{x}$
        gives an isomorphism from~$[0,1]$ to~$\mathrm{Stat}\C^2$.
        Thus $F'\C^2 \cong C[0,1]$.
        Hence on~$\C^2$
        the functor~$F$ and its commutative variant~$F'$ agree
        (see Example~\ref{ex:C2}).
        However,
        on~$\C^3$ the functors~$F$ and~$F'$ differ.
        Indeed, $F'\C^3$ is commutative
        while $F\C^3$ is not (see Example~\ref{ex:C3}).
        \begin{equation*}
                \xymatrix{
        \CCMIU
        \ar[d]
        \ar@/_1em/[r]
        \ar@{}[r]|{\rotatebox[origin=c]{90}{$\vdash$}}
        &
        \CCPU
        \ar[d]
        \ar@/_1em/[l]_{F'}
        \\
        \CMIU
        \ar@/_1em/[r]
        \ar@{}[r]|{\rotatebox[origin=c]{90}{$\vdash$}}
        &
        \CPU
        \ar@/_1em/[l]_{F}
        }
        \end{equation*}
        Roughly summarised:
        while
        in the diagram above
        the right adjoints commute with
        the vertical embeddings,
        the left adjoints do not.
\end{rem}
\begin{thm}
        \label{thm:left-adjoint}
        The embedding $U\colon \CMIU \longrightarrow \CPU$
        has a left adjoint.
\end{thm}
\begin{proof}
        By Freyd's Adjoint Functor Theorem
        (see Theorem~V.6.1 of~\cite{maclane1998})
        and the fact that all limits can be formed using only products
        and equalisers 
        (see Theorem~V.2.1 and Exercise~V.4.2 of~\cite{maclane1998})
        it suffices to prove the following.
        \begin{enumerate}
                \item
                        \label{freyd-cond:i}
                        The category $\CMIU$ has all small products 
                        and equalisers.
                \item
                        \label{freyd-cond:ii}
                        The functor $U\colon \CMIU\longrightarrow \CPU$
                        preserves small products and equalisers.
                \item
                        \label{freyd-cond:iii}
                        \emph{Solution Set Condition.}\ 
                        For every $C^*$-algebra~$\mathscr{A}$
                        there is a set~$I$
                        and for each~$i\in I$
                        a PU-map $f_i\colon \mathscr{A}\ra 
                        \mathscr{A}_i$
                        such that for any PU-map $f\colon \mathscr{A}\ra
                        \mathscr{B}$
                        there is an~$i\in I$
                        and a
                        MIU-map $h\colon \mathscr{A}_i\ra\mathscr{B}$
                        such that $h\circ f_i = f$.
        \end{enumerate}
        Conditions~\ref{freyd-cond:i}
        and~\ref{freyd-cond:ii} can be verified with routine
        so we will spend only a few words on them (and leave
        the details to the reader).
        To see that Condition~\ref{freyd-cond:iii} holds
        requires
        a little more ingenuity and so we will
        give the proof in detail.
        \vspace{.4em}

        \noindent
        \emph{(Conditions~\ref{freyd-cond:i} and~\ref{freyd-cond:ii})}\ 
        Let us first think about small products
        in~$\CMIU$ and~$\CPU$.

        Let $I$ be a set, 
        and for each~$i\in I$ let $\mathscr{A}_i$ be a $C^*$-algebra.

        It is not
        hard to see that cartesian product $\prod_{i\in I}\mathscr{A}_i$
        is a $*$-algebra when endowed with coordinate-wise operations
        (and it is in fact the product of the $\mathscr{A}_i$
        in the category of $*$-algebras
        with MIU-maps, and with PU-maps).

        However,
        $\prod_{i\in I}\mathscr{A}_i$
        cannot be the product of the~$\mathscr{A}_i$
        as $C^*$-algebras:
        there is not even a $C^*$-norm
        on $\prod_{i\in I}\mathscr{A}_i$
        unless~$\mathscr{A}_i$ is trivial for 
        all but finitely many $i\in I$.
        Indeed, if $\|-\|$ were a $C^*$-norm on~$\prod_{i\in I}\mathscr{A}_i$,
        then we must have $\|\sigma(i)\|\leq \|\sigma\|$
        for all~$\sigma\in \prod_{i\in I}\mathscr{A}_i$
        and~$i\in I$,
        and 
        so for any sequence $i_0,\,i_1,\,\dotsc $
        of distinct elements of~$I$
        for which~$\mathscr{A}_{i_0},\,\mathscr{A}_{i_1},\,\dotsc$
        are non-trivial,
        and for every $\sigma\in \prod_{i\in I} \mathscr{A}_i$
        with $\sigma(i_n)=n\cdot 1$ for all~$n$,
        we have $n = \|\sigma(i_n)\|\leq \|\sigma\|$ for all~$n$,
        so $\|\sigma\|=\infty$,
        which is not allowed.

        Nevertheless,
        the $*$-subalgebra
        of~$\prod_{i\in I}\mathscr{A}_i$
        given by
        \begin{equation*}
                \textstyle{\bigoplus_{i\in I}\mathscr{A}_i
                \ \eqdf\ 
                \{\ \sigma\in \prod_{i \in I} \mathscr{A}_i\colon\ 
        \sup_{i\in I}\|\sigma(i)\| \,<\,+\infty\ \}}
        \end{equation*}
        is a $C^*$-algebra with norm given by,
        for $\sigma\in\bigoplus_{i\in I}\mathscr{A}_i$,
        \begin{equation*}
                \| \sigma \|\ =\ \textstyle{\sup_{i\in I}\|\sigma(i)\|}.
        \end{equation*}
        We claim that $\bigoplus_{i\in I}\mathscr{A}_i$
        is the product of the~$\mathscr{A}_i$
        in~$\CPU$ (and in~$\CMIU$).
        
        Let $\mathscr{C}$ be a $C^*$-algebra,
        and for each~$i\in I$,
        let $f_i\colon \mathscr{C}\ra \mathscr{A}_i$
        be a PU-map.
        We must show that there is a unique PU-map
        $f\colon \mathscr{C}\ra \bigoplus_{i\in I}\mathscr{A}_i$
        such that $\pi_i\circ f = f_i$ for all~$i\in I$
        where $\pi_i\colon \bigoplus_{j\in I} \mathscr{A}_j
        \ra \mathscr{A}_i$
        is the $i$-th projection.
        It is clear that there is at most one such~$f$,
        and it would satisfy
        for all~$i\in I$, and~$c\in\mathscr{C}$,
        $f(c)(i)\ = \ f_i(c)$.

        To see that such map~$f$ exists is easy if
        we are able to prove that,
        for all~$c\in\mathscr{C}$,
        \begin{equation}
                \label{eq:freyd-sup}
                \textstyle{\sup_{i\in I}} \|f_i(c)\|\ <\ +\infty.
        \end{equation}
        Let $i\in I$ be given.
        We claim that that~$\|f_i(c)\|\leq \|c\|$
        for any \emph{positive} $c \in\mathscr{C}$.
        Indeed,
        we have $c\leq \|c\|\cdot 1$,
        and thus
        $f_i(c) \leq 
        \|c\|\cdot f(1)=
        \|c\|\cdot 1$, 
        and so $\|f_i(c)\|\leq \|c\|$.
        It follows that $\|f_i(c)\| \leq 4 \cdot\|c\|$
        for any~$c\in \mathscr{A}$
        by writing $c = c_1 - c_2 + ic_3 - ic_4$
        where $c_1,\,c_2,\,c_3,\,c_4\in\mathscr{C}$
        are all positive.  (We even have
        $\|f(c)\|\leq \|c\|$ for all~$c\in\mathscr{C}$,
        but this requires a bit more effort%
        \footnote{See Corollary~1 of~\cite{russo1966}.}) Thus, we have
        $\sup_{i\in I} \|f_i(c)\| \leq 4\|c\| < + \infty$.
        Hence Statement~\eqref{eq:freyd-sup} holds.

        Thus $\bigoplus_{i\in I}\mathscr{A}_i$
        is the product of the~$\mathscr{A}_i$ in~$\CPU$.
        It is easy to see that $\bigoplus_{i\in I}\mathscr{A}_i$
        is the product of the~$\mathscr{A}_i$
        in~$\CMIU$ as well.
        Hence~$\CMIU$ has all small products (as does~$\CPU$)
        and~$U\colon \CMIU\longrightarrow \CPU$
        preserves small products.

        \vspace{.4em}
        Let us think about equalisers in~$\CMIU$
        and~$\CPU$.
        Let $\mathscr{A}$ and~$\mathscr{B}$ be $C^*$-algebras
        and let $f,g\colon \mathscr{A}\rightarrow \mathscr{B}$
        be MIU-maps.
        We must prove that~$f$ and~$g$ have an 
        equaliser~$e\colon \mathscr{E}\ra \mathscr{A}$
        in~$\CMIU$, and that~$e$ is the equaliser of~$f$ and~$g$
        in~$\CPU$ as well.

        Since $f$ and~$g$ are MIU-maps
        (and hence continuous),
        it is not hard to see that
        \begin{equation*}
                \mathscr{E} \ \eqdf\ 
                \{\ a\in\mathscr{A}\colon\ f(a)\,=\,g(a)\ \}
        \end{equation*}
        is a $C^*$-subalgebra of~$\mathscr{A}$.
        We claim that the inclusion $e\colon \mathscr{E}\ra \mathscr{A}$
        is the equaliser of~$f,g$ in~$\CPU$.
        Let~$\mathscr{D}$ be a $C^*$-algebra
        and let $d\colon \mathscr{D}\ra \mathscr{A}$
        be a PU-map such that~$f\circ d = g\circ d$.
        We must show that there is a unique PU-map 
        $h\colon \mathscr{D}\ra \mathscr{E}$
        such that~$d = e\circ h$.
        Note that~$d$ maps~$\mathscr{A}$ into~$\mathscr{E}$.
        The map~$h\colon \mathscr{D}\ra \mathscr{E}$ is simply 
        the restriction of~$d\colon \mathscr{D}\ra \mathscr{A}$
        in the codomain. 
        Hence~$e$ is the equaliser of~$f,g$ in~$\CPU$.
        
        Note that in the argument above~$h$ is a PU-map
        since~$d$ is a PU-map.
        If~$d$ were a MIU-map,
        then~$h$ would be a MIU-map too.
        Hence~$e$ is the equaliser of~$f,g$ in the category~$\CMIU$
        as well.

        Hence $\CMIU$ has all equalisers
        and~$U\colon \CMIU\longrightarrow \CPU$ preserves equalisers.
        Hence~$\CMIU$ has all small limits
        and~$U\colon \CMIU\longrightarrow \CPU$ preserves all small limits.

        (Note that while we have seen that~$\CPU$
        has all small products,
        and it was easy to see that~$\CMIU$
        has all equalisers,
        it is not clear whether~$\CPU$
        has all equalisers.
        Indeed, if $f,g\colon \mathscr{A}\ra \mathscr{B}$
        are PU-maps,
        then the set $\{ a\in \mathscr{A}\colon f(a)=g(a)\}$
        need not be a $C^*$-subalgebra of~$\mathscr{A}$.)
        \vspace{.4em}

        \noindent
        \emph{(Condition~\ref{freyd-cond:iii}).}\ 
        Let~$\mathscr{A}$ be a $C^*$-algebra.
        We must find a set~$I$
        and for each~$i\in I$ a PU-map 
        $f_i\colon \mathscr{A}\ra \mathscr{A}_i$
        such that for every PU-map
        $f\colon \mathscr{A}\ra \mathscr{B}$
        there is a (not necessarily unique) 
        $i\in I$ and $h\colon \mathscr{A}_i\ra \mathscr{B}$
        such that $f=h\circ f_i$.

        Note that if $f\colon \mathscr{A}\ra \mathscr{B}$
        is a PU-map,
        then the range of the PU-map~$f$ need not be a $C^*$-subalgebra
        of~$\mathscr{B}$.
        (If the range of PU-maps would have been $C^*$-algebras,
        then we could have taken $I$ to be the set of all ideals 
        of~$\mathscr{A}$,
        and $f_J\colon \mathscr{A}\ra \mathscr{A}/J$
        to be the quotient map for any ideal~$J$ of~$\mathscr{A}$.)

        Nevertheless,
        given a PU-map $f\colon \mathscr{A}\ra \mathscr{B}$
        there is a smallest $C^*$-subalgebra, say $\mathscr{B}'$,
        of~$\mathscr{B}$ that contains the range of~$f$.
        \textbf{We claim that~$\#\mathscr{B}' \leq \#(\mathscr{A}^\N)$}
        where 
        $\#\mathscr{B}'$ is the cardinality of~$\mathscr{B}'$
        and 
        $\#(\mathscr{A}^\N)$ is the cardinality of~$\mathscr{A}^\N$.\footnote{%
                Although it has no bearing on the validity of the proof
                one might wonder if the simpler statement 
                $\# \mathscr{B}'\leq \#\mathscr{A}$ holds as well.
                Indeed, if $\#\mathscr{A} = \#\C$ or
                $\#\mathscr{A}=\#(2^X)$ for some infinite set~$X$,
                then we have $\#\mathscr{A} = \#(\mathscr{A}^\N)$,
                and so 
                $\# \mathscr{B}'\leq \#\mathscr{A}$.
                However,
                not every uncountable set is of the form
                $2^X$ for some infinite set~$X$,
                and in fact,
                if $\#\mathscr{A} = \aleph_\omega$,
                then 
                $\#(\mathscr{A}^\N) > \#\mathscr{A}$
                by Corollary~3.9.6 of~\cite{devlin1993}}

        If we can find proof for our claim,
        the rest is easy.
        Indeed, to begin
        note that the collection
        of all $C^*$-algebras is not a small set.
        However, given a set~$U$, 
        the collection of all $C^*$-algebras~$\mathscr{C}$
        whose elements come from~$U$ (so $\mathscr{C}\subseteq U$)
        is a small set.
        Now,
        let~$\kappa\eqdf \# (\mathscr{A}^\N)$
        be the cardinality of~$\mathscr{A}^\N$
        (so~$\kappa$ is itself a set)
        and take
        \begin{equation*}
                I \ \eqdf\ 
                \{\ (\mathscr{C},c)\colon\ 
                        \text{$\mathscr{C}$ is a $C^*$-algebra
                        on a subset of~$\kappa$
                        and $c\colon \mathscr{A}\ra \mathscr{C}$
        is a PU-map} \ \}.
        \end{equation*}
        Since the collection of $C^*$-algebras $\mathscr{C}$
        with $\mathscr{C}\subseteq \kappa$ is small,
        and  since the collection of PU-maps from~$\mathscr{A}$
        to~$\mathscr{C}$ is small for any $C^*$-algebra~$\mathscr{C}$,
        it follows that~$I$ is small.

        For each $i\in I$ with  $i\equiv(\mathscr{C},c)$
        define $\mathscr{A}_i\eqdf \mathscr{C}$
        and $f_i \eqdf c$.

        Let $f\colon \mathscr{A}\ra \mathscr{B}$
        be a PU-map.
        We must find $i\in I$ and a MIU-map 
        $h\colon \mathscr{A}_i\ra \mathscr{B}$
        such that $h\circ f_i = f$.
        Let~$\mathscr{B}'$ be the smallest $C^*$-subalgebra
        that contains the range of~$f$.
        By our claim we have 
        $\# \mathscr{B}' \leq \#(\mathscr{A}^\N) \equiv\kappa$.
        By renaming the elements of~$\mathscr{B}'$
        we can find a $C^*$-algebra~$\mathscr{C}$
        isomorphic to~$\mathscr{B}'$
        whose elements come from~$\kappa$.
        Let $\varphi\colon \mathscr{C}\ra\mathscr{B}'$
        be the isomorphism.

        Note that $c\eqdf \varphi^{-1}\circ f\colon \mathscr{A}\ra \mathscr{C}$
        is a PU-map.
        So we have $i\eqdf (\mathscr{C},c)\in I$.
        Further, the inclusion 
        $e\colon \mathscr{B}'\rightarrow \mathscr{B}$
        is a MIU-map,
        as is~$\varphi$. So we have:
        \begin{equation*}
        \xymatrix{
                \mathscr{A}
                \ar[rr]^f_{\mathrm{PU}}
                \ar[d]_{c}^{\mathrm{PU}}
                &&
                \mathscr{B}
                \\
                \mathscr{C}
                \ar[rr]_\varphi^{\mathrm{MIU}}
                &&
                \mathscr{B}'
                \ar[u]_{e}^{\mathrm{MIU}}
        }
        \end{equation*}
        Now, $h\eqdf e\circ \varphi\colon \mathscr{C}\ra \mathscr{B}$ is 
        a MIU-map with $f = h\circ f_i$.
        Hence Cond.~\ref{freyd-cond:iii} holds.
        \vspace{.4em}

        Let us proof our claim.
        Let $\mathscr{A}$ and~$\mathscr{B}$
        be $C^*$-algebras
        and let $f\colon \mathscr{A}\ra \mathscr{B}$
        be a PU-map.
        Let~$\mathscr{B}'$ be the smallest $C^*$-subalgebra
        that contains the range of~$f$.\\
        We must show that $\#\mathscr{B}'\leq \#(\mathscr{A}^\N)$.
        
        Let us first take  care of pathological case.
        Note that if $\mathscr{A}$ is trivial,
        i.e.~$\mathscr{A}=\{0\}$,
        then~$\mathscr{B}' =\{0\}$,
        so $ \#(\mathscr{A}^\N)= 1 =  \#\mathscr{B}'$.
        Now,
        let us assume that~$\mathscr{A}$ is not trivial.
        Then we have an injection $\C\ra \mathscr{A}$
        given by $\lambda\mapsto \lambda \cdot 1$,
        and thus $\#\C\leq\#\mathscr{A}$.

        The trick to prove~$\#\mathscr{B}'\leq \#(\mathscr{A}^\N)$
        is to find a more explicit description of~$\mathscr{B}'$.
        Let~$T$ be the set of terms
        formed using a unary operation~$(-)^*$ 
        (involution)
        and two binary operations, ~$\cdot$ (multiplication)
        and $+$ (addition), 
        starting from the elements of~$\mathscr{A}$.
        Let $f_T\colon T\longrightarrow \mathscr{B}'$
        be the map (recursively) given by,
        for $a\in \mathscr{A}$, and $s,t\in T$,
        \begin{alignat*}{3}
                f_T(a)\ &=\ f(a)\\
                f_T(s^*) \ &=\ (f_T(s))^* \\
                f_T(s\cdot t) \ &= \ f_T(s)\,\cdot\, f_T(t)\\
                f_T(s+t)\ &=\ f_T(s)\,+\,f_T(t).
        \end{alignat*}
        Note that the range of~$f_B$,
        let us call it~$\Ran{f_B}$,
        is a $*$-subalgebra of~$\mathscr{B}'$.
        We will prove that~$\#\Ran{f_B} \leq \#\mathscr{A}$.
        Since~$f_B$ is a surjection of~$T$ onto~$\Ran{f_B}$
        it suffices to prove that~$\#T \leq \# \mathscr{A}$.
        In fact,
        we will show that~$\#T = \#\mathscr{A}$.

        First note that~$\mathscr{A}$ is infinite,
        and $\mathscr{A}\subseteq T$,
        so~$T$ is infinite as well.
        To prove that $\#T = \#\mathscr{A}$
        we write the elements of~$T$ 
        as words (with the use of brackets).
        Indeed,
        with $Q\eqdf \mathscr{A}\cup \{ 
        \text{``}\cdot\text{''},
        \text{``}+\text{''},
        \text{``}*\text{''},
        \text{``})\text{''},
        \text{``}(\text{''} \}$
        there is an obvious injection from~$T$ into
        the set~$Q^*$ of words over~$Q$.
        Since~$\mathscr{A}$ is infinite,
        and~$Q \backslash\mathscr{A}$ is finite
        we have $\#Q = \#\mathscr{A}$
        by Hilbert's hotel.
        Recall that $Q^* = \bigcup_{n=0}^\infty Q^n$.
        Since~$Q$ is infinite,
        we also have $\#(\N\times Q)=\#Q$
        and even  $\#(Q\times Q)=\#Q$
        (see Theorem~3.7.7 of~\cite{devlin1993}),
        so $\# Q = \# (Q^n)$ for all~$n>0$.
        It follows that 
        \begin{alignat*}{3}
                \# (Q^*) 
                \ &= \ 
                \#(\, \textstyle{\bigcup_{n=0}^\infty} Q^n\, ) \\
                \ &= \ 
                \#(\, 1+ \textstyle{\bigcup_{n=1}^\infty} Q\, ) \\
                \ &= \ 
                \#(\, 1+ \N\times Q\, ) \\
                \ &= \ 
                \#Q.
        \end{alignat*}
        Since there is an injection from~$T$ to~$Q^*$ we have
        $\# \mathscr{A} \leq \# T \leq \# (Q^*) = \# Q =\# \mathscr{A}$
        and so $\#T = \#\mathscr{A}$.
        Hence $\#\Ran{f_B} \leq \#\mathscr{A}$.

        Since~$\Ran{f_B}$
        is a $*$-algebra that contains~$\Ran{f}$,
        the closure $\smash{\overline{\Ran{f_B}}}$
        of~$\Ran{f_B}$ with respect to the norm on~$\mathscr{B}'$
        is a $C^*$-algebra that contains~$\Ran{f}$.
        As~$\mathscr{B}'$ is the smallest $C^*$-subalgebra
        that contains~$\Ran{f}$,
        we see that~$\mathscr{B}'=\smash{\overline{\Ran{f_B}}}$.
        
        Let $S$ be the set of all Cauchy sequences
        in~$\Ran{f_B}$.
        As every point in~$\mathscr{B}'$
        is the limit of a Cauchy sequence in~$\Ran{f_B}$,
        we get $\# \mathscr{B}' \leq \# S$.
        Thus:
        \begin{alignat*}{3}
                \#\mathscr{B'} 
                \ &\leq\ 
                \# S &&\\
                \ &\leq\ 
                \# \,(\Ran{f_B})^\N \qquad
                && \text{as $S\subseteq (\Ran{f_B})^\N$} \\
                \ &\leq\ 
                \#(\, \mathscr{A}^\N\,)\qquad
                &&\text{as $\# \Ran{f_B}\leq \#\mathscr{A}$}.
        \end{alignat*}
        Thus we have proven our claim.
        
        Hence Conditions~\ref{freyd-cond:i}--\ref{freyd-cond:iii} hold
        and $U\colon \CMIU \longrightarrow \CPU$
        has a left adjoint.
\end{proof}

We have seen that $U\colon \CMIU \longrightarrow \CPU$
has a left adjoint~$F\colon \CPU\longrightarrow \CMIU$.
This adjunction gives a comonad~$FU$ on~$\CMIU$,
which in turns gives us two categories:
the Eilenberg--Moore category $\EM{FU}$ of~$FU$-coalgebras
and the Kleisli category~$\Kl{FU}$. 
We claim that~$\CPU$ is isomorphic to~$\Kl{FU}$
since~$\CMIU$ is a subcategory of~$\CPU$ with the same objects.

This is a special case of a more general phenomenon
which we discuss 
in the next section
(in terms of monads instead of comonads), 
see Theorem~\ref{thm:kleisli-beck}.

\section{Kleislian Adjunctions}
\label{S:kleisli}
Beck's Theorem (see \cite{maclane1998}, VI.7) gives a criterion for when
an adjunction $F\dashv U$ ``is'' an adjunction between $\Cat{C}$
and~$\EM{UF}$.
We give a similar (but easier) criterion
for when an adjunction ``is''
an adjunction between $\Cat{C}$ and~$\Kl{UF}$.
The criterion is not new;
e.g., it is mentioned in~\cite{lack2010} (paragraph 8.6)
without proof or reference,
and it can be seen as a consequence of Exercise~VI.5.2 of~\cite{maclane1998}
(if one realises
that an equivalence which is bijective on objects is
an isomorphism).
Proofs can be found in the appendix.
\begin{nt}
        \label{nt:beck}
        Let 
        $F\colon \Cat{C}\longrightarrow \Cat{D}$
        be a functor with right adjoint~$U$.
        Denote the unit of the adjunction
        by $\eta\colon \id{\Cat{D}} \ra UF$,
        and the counit by $\varepsilon\colon FU\ra \id{\Cat{C}}$.

        Recall that~$UF$ is a monad
        with unit~$\eta$
        and as multiplication, for $C$ from~$\Cat{C}$,
        \begin{equation*}
                \mu_C \ \eqdf\  U\varepsilon_{FC} 
        \colon \ UFUFC\longrightarrow UFC.
        \end{equation*}

        Let~$\Kl{UF}$ be the Kleisli category of the monad~$UF$.
        So~$\Kl{UF}$
        has the same objects as~$\Cat{C}$,
        and the morphisms in~$\Kl{UF}$ from~$C_1$ to~$C_2$
        are the morphism in~$\Cat{C}$ from~$C_1$ to~$UFC_2$.
        Given~$C$ from~$\Cat{C}$
        the identity in~$\Kl{UF}$ on~$C$ is~$\eta_C$.
        If $C_1,C_2,C_3$, $f\colon C_1 \ra C_2$,
         $g\colon C_2 \ra C_3$
        from~$\Cat{C}$
        are given,
        $g$ after~$f$ in~$\Kl{UF}$ is
        \begin{equation*}
                g \afterKl f
        \ \eqdf\ 
        \mu_{C_3} \circ UFg \circ f. 
        \end{equation*}

        Let $V\colon \Cat{C}\longrightarrow \Kl{UF}$
        be given by,
        for $f\colon C_1\longrightarrow C_2$
        from~$\Cat{C}$,
        \begin{equation*}
        Vf \ \eqdf \ \eta_{C_2} \circ f\colon \quad
        C_1 \longrightarrow UFC_2.
        \end{equation*}

        Let $G\colon \Kl{UF}\longrightarrow \Cat{C}$
        be given by,
        for
        $f\colon C_1\longrightarrow UFC_2$
        from~$\Cat{C}$,
        \begin{equation*}
        Gf \ \eqdf \ \mu_{C_2} \circ UF f\colon \quad
        UFC_1 \longrightarrow UFC_2.
        \end{equation*}
\end{nt}

\noindent 
The following is  Exercise VI.5.1 of~\cite{maclane1998}.
\begin{lem}
        \label{lem:L}
        Let $F\colon \Cat{C}\longrightarrow \Cat{D}$
        be a functor with a right adjoint~$U$.\\
        Then there is a unique functor
                $L\colon \Kl{UF}\longrightarrow\Cat{D}$
        (called the \emph{comparison functor})
        such that $U\circ L=G$
        and $L\circ V=F$
        (see Notation~\ref{nt:beck}).
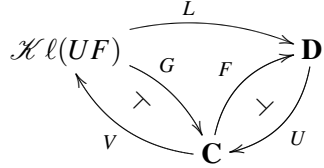
        \begin{equation*}
        \xymatrix{
                \Kl{UF}
                \ar@/^1em/[rd]^G
                \ar@/^1em/[rr]^L
                \ar@{}[rd]|{\rotatebox[origin=c]{225}{$\vdash$}}
                &
                &
                \Cat{D}
                \ar@/^1em/^U[ld]
                \\
                &
                \Cat{C}
                \ar@/^1em/[ul]^V
                \ar@{}[ru]|{\rotatebox[origin=c]{135}{$\vdash$}}
                \ar@/^1em/[ru]^F
                &
        }
        \end{equation*}
\end{lem}

\begin{dfn}
        \label{def:kleisli}
        Let $\Cat{C}$ and~$\Cat{D}$ be categories.
        \begin{enumerate}
                \item
                A functor $F\colon \Cat{C}\longrightarrow \Cat{D}$
                is called \emph{Kleislian}
                when it has a right adjoint~$U\colon \Cat{D}\to\Cat{C}$,
                and the functor~$L\colon \Kl{UF}\longrightarrow \Cat{D}$
                from Lemma~\ref{lem:L} is an isomorphism.

                \item
                        We say that $\Cat{D}$ \emph{is Kleislian over}~$\Cat{C}$
                        when there is a Kleislian
                        functor~$F\colon \Cat{C}\longrightarrow\Cat{D}$.
        \end{enumerate}
\end{dfn}

\begin{thm}
        \label{thm:kleisli-beck}
        Let $F\colon \Cat{C}\longrightarrow \Cat{D}$
        be a functor with a right adjoint~$U$.\\
        The following are equivalent.
        \begin{enumerate}
        \item
                \label{thm:kleisli-beck-i}
                $F$ is Kleislian (see Definition~\ref{def:kleisli}).
        \item
                \label{thm:kleisli-beck-ii}
                $F$ is bijective on objects
                (i.e. for every object~$D$ from~$\Cat{D}$
                there is a unique object~$C$ from~$\Cat{C}$
                such that~$FC=D$).
        \end{enumerate}
\end{thm}
\begin{cor}
\label{C:kleisli}
The embedding $U^\mathrm{op}\colon (\CMIU)^\mathrm{op} \longrightarrow (\CPU)^\mathrm{op}$
is Kleislian
(see Def.~\ref{def:kleisli}).
\end{cor}
\begin{proof}
By Theorem~\ref{thm:kleisli-beck}
we must show that~$U^\mathrm{op}$
has a left adjoint and is bijective on objects.
Since the embedding $U\colon \CMIU \ra \CPU$
has a \emph{left} adjoint~$F\colon \CPU\ra \CMIU$
it follows that $F^\mathrm{op}\colon (\CPU)^\mathrm{op}\ra(\CMIU)^\mathrm{op}$
is the \emph{right} adjoint of~$U^\mathrm{op}$.
Thus~$U^\mathrm{op}$ has a left adjoint.
Further, 
as $\CMIU$ and~$\CPU$ 
have the same objects,
$U$ is bijective on objects,
and so is~$U^\mathrm{op}$.
Hence $U^\mathrm{op}$ is Kleislian.
\end{proof}
In summary,
the embedding $U\colon \CMIU\longrightarrow \CPU$
has a left adjoint~$F$
(and so~$F^\mathrm{op}\colon (\CMIU)^\mathrm{op} \ra (\CPU)^\mathrm{op}$
is \emph{right} adjoint to~$U^\mathrm{op}$),
and 
the unique functor from the Kleisli category~$\Kl{FU}$ 
of the monad~$FU$ on~$(\CMIU)^\mathrm{op}$ to~$(\CPU)^\mathrm{op}$
that makes the two triangles in the diagram below on the left commute
is an isomorphism.
        
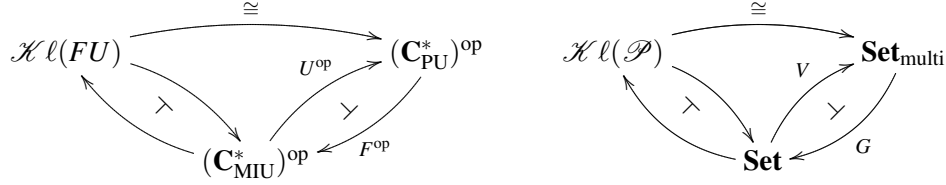
\begin{equation*}
        \xymatrix{
                \Kl{FU}
                \ar@/^1em/[rd]
                \ar@/^1em/[rr]^{\cong}
                \ar@{}[rd]|{\rotatebox[origin=c]{225}{$\vdash$}}
                &
                &
                (\CPU)^\mathrm{op}
                \ar@/^1em/^{F^\mathrm{op}}[ld]
                & 
                \Kl{\mathcal{P}}
                \ar@/^1em/[rd]
                \ar@/^1em/[rr]^{\cong}
                \ar@{}[rd]|{\rotatebox[origin=c]{225}{$\vdash$}}
                &
                &
                \Cat{Set}_\mathrm{multi}
                \ar@/^1em/^{G}[ld]
                \\
                &
                (\CMIU)^\mathrm{op}
                \ar@/^1em/[ul]
                \ar@{}[ru]|{\rotatebox[origin=c]{135}{$\vdash$}}
                \ar@/^1em/[ru]^{U^\mathrm{op}}
                &
                &
                &
                \Cat{Set}
                \ar@/^1em/[ul]
                \ar@{}[ru]|{\rotatebox[origin=c]{135}{$\vdash$}}
                \ar@/^1em/[ru]^{V}
                &%
        }
        \end{equation*}
        For the category~$\Cat{Set}_\mathrm{multi}$ of
        multimaps between sets
        used in the introduction
        to describe the semantics of non-deterministic programs
        the situation is the same,
        see the diagram above to the right.

        (The functor~$V$ is the obvious embedding.
The right adjoint~$G$ of~$V$ sends a multimap~$f$ from~$X$ to~$Y$
to the function~$Gf\colon \mathcal{P}(X)\ra\mathcal{P}(Y)$
that assigns to a subset~$A\in \mathcal{P}(X)$
the image of~$A$ under~$f$.
Note that~$GV=\mathcal{P}$.)

\section{Discussion}
\label{S:discussion}
\subsection{Variations}
\begin{ex}[Subunital maps]
        Let $\CPsU$ be the category of $C^*$-algebras
        and the positive linear maps~$f$ between them
        that are \emph{subunitial},
        i.e.~$f(1)\leq 1$.
        The morphisms of~$\CPsU$ are
        called \emph{PsU-maps}.

        It is not hard to see that
        the products in~$\CPsU$ are the same as in~$\CMIU$,
        and that the equaliser in~$\CMIU$ of a pair~$f,g$ of MIU-maps
        is the equaliser of~$f,g$ in~$\CPsU$ as well.
        Thus the embedding $U\colon \CMIU\longrightarrow \CPsU$
        preserves limits.
        Using the same argument as in Theorem~\ref{thm:left-adjoint}
        but with ``PU-map'' replaced by ``PsU-map''
        one can show that $U$ satisfies the Solution Set Condition.
        Hence~$U$ has a left adjoint
        by Freyd's Adjoint Function Theorem, 
        say $F\colon \CPsU\longrightarrow \CMIU$.

        Since~$\CPsU$ has the same objects as~$\CMIU$
        (namely the $C^*$-algebras)
        the functor $U^\mathrm{op}\colon 
        (\CMIU)^\mathrm{op}\longrightarrow (\CPsU)^\mathrm{op}$
        is bijective on objects and thus Kleislian
        (by Th.~\ref{thm:kleisli-beck}).

        Hence $(\CPsU)^\mathrm{op}$ is Kleislian over~$(\CMIU)^\mathrm{op}$.
\end{ex}

\begin{ex}[Bounded linear maps]
        Let~$\CP$ be the category
        of positive bounded linear maps between $C^*$-algebras.
        We will show that~$(\CP)^\mathrm{op}$ is \emph{not}
        Kleislian over~$(\CMIU)^\mathrm{op}$.
        Indeed, if it were
        then~$(\CP)^\mathrm{op}$ would be cocomplete,
        but it is not: there is no $\omega$-fold product
        of~$\C$ in~$\CP$.
        To see this, suppose that there is
        a $\omega$-fold product~$\mathscr{P}$
        in~$\CP$ with projections $\pi_i\colon \mathscr{P}\ra \C$
        for~$i\in \omega$.
        Since~$\pi_i$ is a bounded linear map
        for~$i\in \omega$,
        it has finite operator norm, say~$\|\pi_i\|$.
        By symmetry, $\|\pi_i\| = \|\pi_j\|$ for all~$i,j\in \omega$.
        Write
        $K\eqdf \|\pi_0\|=\|\pi_1\| = \|\pi_2\|= \dotsb$.
        Define $f_i\colon \C\ra \C$ by~$f_i(z)=iz$
        for all~$z\in\C$ and~$i\in \omega$.
        Then~$f_i$ is a positive bounded linear map for each~$i\in\omega$.
        Since~$\mathscr{P}$ is the $\omega$-fold product of~$\C$,
        there is a (unique positive) bounded linear 
        map $f\colon \C\ra \mathscr{P}$
        such that $\pi_i \circ f = f_i$ for all~$i \in \omega$.
        For each~$N\in \omega$ we have 
        \begin{equation*}
                N\ =\ \|f_N(1)\| \ \leq\  \|f_N\| \ =\ \|\pi_N\circ f\|
        \ \leq\  \|\pi_N\| \,\|f\|
        = K \,\|f\|.
        \end{equation*}
        Thus $K \|f\|$ is greater than any number, which is absurd.
\end{ex}

\begin{ex}[Completely positive maps]
        For clarity's sake
        we recall what it means 
        for a linear map~$f$ between $C^*$-algebras
        to be completely positive (see~\cite{stinespring1955}).
        For this we need some notation.
        Given a $C^*$-algebra $\mathscr{A}$,
        and $n\in\N$
        let~$M_n(\mathscr{A})$ denote
        the set of $n\times n$-matrices
        with entries from~$\mathscr{A}$.
        We leave it to the reader to check
        that~$M_n(\mathscr{A})$
        is a $*$-algebra with the obvious
        operations.
        In fact, it turns out that $M_n(\mathscr{A})$ is a $C^*$-algebra,
        but some care must be taken to define the norm on~$M_n(\mathscr{A})$
        as we will see below.
        Now,
        a linear map $f\colon \mathscr{A}\longrightarrow\mathscr{B}$
        is called \emph{completely positive}
        when $M_n f$ is positive for each~$n\in \N$,
        where $M_n f\colon M_n(\mathscr{A})\longrightarrow
        M_n(\mathscr{B})$ is the
        map obtained by applying~$f$ to each entry
        of a matrix in~$M_n(\mathscr{A})$.
        Of course,
        ``$M_n f$ is positive''
        only makes sense once we 
        know that $M_n(\mathscr{A})$ and~$M_n(\mathscr{B})$
        are $C^*$-algebras.

        Let~$\mathscr{A}$ be a $C^*$-algebra.
        We will put a $C^*$-norm on~$M_n(\mathscr{A})$.
        Let~$\mathscr{H}$
        be a Hilbert space
        and let
        $\pi\colon \mathscr{A}\longrightarrow \mathscr{B}(\mathscr{H})$,
        be an isometric MIU-map.
        We get a norm~$\|-\|_\pi$ on~$M_n(\mathscr{A})$
        given by for~$A\in M_n(\mathscr{A})$,
        \begin{equation*}
                \| A\|_\pi \ =\  \|\xi(\,(M_n\pi)(A)\,)\|,
        \end{equation*}
        where $\xi(\,(M_n \pi)(A)\,)\colon \mathscr{H}^{\oplus n} \rightarrow
        \mathscr{H}^{\oplus n}$
        is the bounded linear map represented by
        the matrix~$(M_n\pi)(A)$,
        and $\|\xi(\,(M_n\pi)(A)\,)\|$ is the operator
        norm of~$\xi(\,(M_n\pi)(A)\,)$ in~$\mathscr{B}(\mathscr{H}^{\oplus n})$.

        It is easy to see that~$\|-\|_\pi$
        satisfies the $C^*$-identity, $\|A^*A\|_\pi =\|A\|^2_\pi$
        for all~$A\in M_n(\mathscr{A})$.
        It is less obvious that~$M_n(\mathscr{A})$
        is complete with respect to~$\|-\|_\pi$.
        To see this,
        first note that~$\|A_{ij}\| \leq \|A\|_\pi$
        for all~$i,j$.
        So given a Cauchy sequence $A_1,\,A_2,\,\dotsc$ in~$M_n(\mathscr{A})$
        we can form the entrywise limit~$A$, that is, 
        $A_{ij} = \lim_{m\ra \infty} A_{ij}$.  
        We leave it to the reader to check that~$A_{ij}$ is the limit
        of~$A_1,\,A_2,\,\dotsc$,
        and thus $M_n(\mathscr{A})$ is complete with respect to~$\|-\|_\pi$.
        Hence $M_n(\mathscr{A})$ is a $C^*$-algebra
        with norm~$\|-\|_\pi$.

        The $C^*$-norm~$\|-\|_\pi$
        does not depend on~$\pi$.
        Indeed, let $\mathscr{H}_1$ and~$\mathscr{H}_2$
        be Hilbert spaces
        and let 
        $\pi_1\colon \mathscr{A}\longrightarrow\mathscr{B}(\mathscr{H}_1)$
        and
        $\pi_2\colon \mathscr{A}\longrightarrow\mathscr{B}(\mathscr{H}_2)$
        be isometric MIU-maps;
        we will show that~$\|-\|_{\pi_1} = \|-\|_{\pi_2}$.
        Recall that the norm~$\|-\|_{\pi_i}$ induces
        an order~$\leq_{\pi_i}$ on~$M_n(\mathscr{A})$
        given by $0\leq_{\pi_i} A$ iff
        $\|A-\|A\|_{\pi_i}\|_{\pi_i} \leq \|A\|_{\pi_i}$
        where $A\in M_n(\mathscr{A})$.
        Since $\|A\|_{\pi_i}^2 = \inf\{ \ \lambda\in[0,\infty)\colon\ 
                A^*A  \leq_{\pi_i} \lambda \ \}$
                for all~$A\in M_n(\mathscr{A})$,
        to prove $\|-\|_{\pi_1} = \|-\|_{\pi_2}$
        it suffices to show that the orders
        $\leq_{\pi_1}$ and~$\leq_{\pi_2}$ coincide.
        But this is easy when one recalls
        that~$A\in M_n(\mathscr{A})$ is positive
        iff~$A$ is of the form~$B^*B$ for some~$B\in M_n(\mathscr{A})$.

        The completely positive linear maps that
        preserve the unit are called \emph{CPU-maps}.
        Let~$\CcPU$ be the category of CPU-maps
        between $C^*$-algebras.
        Since $M_n(f)$
        is a MIU-map when~$f$ is a MIU-map
        and a MIU-map is positive,
        we see that any MIU-map is completely positive.
        Thus~$\CMIU$ is a subcategory of~$\CcPU$.
        We claim that~$(\CcPU)^\mathrm{op}$
        is Kleislian over~$(\CMIU)^\mathrm{op}$.
        
        Let us show that~$U$ preserves limits.
        To show that~$U$ preserves equalisers,
        let $f,g\colon \mathscr{A} \longrightarrow \mathscr{B}$
        be MIU-maps.
        Then~$\mathscr{E}\eqdf \{x\in \mathscr{A}\colon f(x)=g(x)\}$
        is a $C^*$-subalgebra of~$\mathscr{A}$
        and the embedding $e\colon \mathscr{E}\ra \mathscr{A}$
        is an isometric MIU-map.
        Then~$e$ is the equalisers of~$f,g$ in~$\CMIU$;
        we will show that~$e$ is the equaliser of~$f,g$ in~$\CcPU$.
        Let $\mathscr{C}$ be a~$C^*$-algebra,
        and let~$c\colon \mathscr{C}\ra \mathscr{A}$
        be a CPU-map such that~$f\circ c = g\circ c$
        Let~$d\colon \mathscr{C}\ra \mathscr{E}$
        be the restriction of~$c$.
        It turns out we must prove that~$d$ is completely positive.
        Let~$n\in\N$ be given.
        We must show that~$M_n d \colon M_n \mathscr{C}\ra M_n\mathscr{E}$
        is positive.
        Note that~$M_n e$ is an injective MIU-map
        and thus an isometry.
        So in order to prove that $M_nd$ is positive
        it suffices to show that $M_ne\circ M_nd = M_n(e\circ d)=M_n c$
        is positive, which it is since~$c$ is completely positive.
        Thus~$e$ is the equaliser of~$f,g$ in~$\CcPU$.
        Hence~$U$ preservers equalisers.

        To show that~$U$ preserves products,
        let~$I$ be a set and for each~$i\in I$
        let~$\mathscr{A}_i$ be a $C^*$-algebra.
        We will show that~$\bigoplus_{i\in I} \mathscr{A}_i$
        is the product of the~$\mathscr{A}_i$ in~$\CcPU$.
        Let~$\mathscr{C}$ be a $C^*$-algebra,
        and for each~$i\in I$,
        let $f_i\colon \mathscr{C}\ra \mathscr{A}_i$
        be a CPU-map.
        As before, let $f\colon \mathscr{C}\ra \bigoplus_{i\in I}{A}_i$
        be the map given by~$f(x)(i) = f_i(x)$ for all~$i\in I$ 
        and~$x\in \mathscr{C}$.
        Leaving the details to the reader
        it turns out that it suffices
        to show that~$f$ is completely positive.
        Let~$n\in \N$ be given.
        We must prove 
        that~$M_n f\colon M_n(\mathscr{C})\longrightarrow 
        M_n(\bigoplus_{i\in I}\mathscr{A}_i)$ is positive.
        Let~$\varphi\colon M_n(\bigoplus_{i\in I} \mathscr{A}_i)
        \longrightarrow \bigoplus_{i\in I} M_n(\mathscr{A}_i)$
        be the unique MIU-map such that
        $\pi_i\circ \varphi = M_n\pi_i$ for all~$i\in I$.
        Then~$\varphi$ is a MIU-isomorphism
        and thus to prove that~$M_nf$ is positive,
        it suffices to show that~$\varphi\circ M_nf$ is positive.
        Let $i\in I$ be given.
        We must prove that $\pi_i\circ \varphi\circ M_nf$ is positive.
        But we have $\pi_i\circ \varphi\circ M_nf =  M_n\pi_i \circ M_nf
        = M_n(\pi_i\circ f) = M_nf_i$,
        which is positive since $f$ is completely positive.
        Thus $\bigoplus_{i\in I}\mathscr{A}_i$
        is the product of the~$\mathscr{A}_i$ in~$\CcPU$
        and hence~$U$ preserves limits.

        With the same argument as in Theorem~\ref{thm:kleisli-beck}
        the functor~$U$ satisfies the Solution Set Condition
        and thus~$U$ has a left adjoint.
        It follows that~$U^\mathrm{op}\colon (\CMIU)^\mathrm{op}
        \longrightarrow (\CcPU)^\mathrm{op}$
        is Kleislian.
\end{ex}

\begin{ex}[$W^*$-algebras]
        Let $\WNMIU$ be the category
        of von Neumann algebras
	(also called $W^*$-algebras)
        and the MIU-maps between them that are normal,
        i.e., preserve suprema of upwards directed
        sets of self-adjoint elements.
        Let $\WNPU$ be the category
        of von Neumann
        and normal PU-maps.
        Note that~$\WNMIU$ is a subcategory of~$\WNPU$.
        We will prove that~$(\WNPU)^\mathrm{op}$
        is Kleislian over~$(\WNMIU)^\mathrm{op}$.

        It suffices to show that~$U$ has a left adjoint.
        Again we follow the lines of the 
        proof of Theorem~\ref{thm:left-adjoint}.
        Products and equalisers in~$\WNMIU$
        are the same as in~$\CMIU$.
        It is not hard to see that
        the embedding $U\colon \WNMIU \longrightarrow \WNPU$
        preserves limits.
        To see that~$U$ satisfies the Solution Set Condition
        we use the same method as before:
        given a von Neumann algebra~$\mathscr{A}$, 
        find a suitable cardinal~$\kappa$
        such that
        the following is a solution set.
        \begin{alignat*}{3}
                I \ \eqdf\ 
                \{\ (\mathscr{C},c)\colon\ 
                        \text{$\mathscr{C}$ is a }&\text{von Neumann algebra
                on a subset of~$\kappa$ }\\
                &\text{and $c\colon \mathscr{A}\longrightarrow
                \mathscr{C}$
        is a normal PU-map} \ \},
        \end{alignat*}
        Only this time we
        take~$\kappa = \# (\,\wp(\wp(\mathscr{A}))\,)$
        instead of~$\kappa = \#(\, \mathscr{A}^\N\,)$.
        We leave the details to the reader,
        but it follows from the fact
        that given a subset~$X$ of a von Neumann algebra~$\mathscr{B}$
        the smallest von Neumann subalgebra~$\mathscr{B}'$ that contains~$X$
        has cardinality at most~$\# (\,\wp(\wp(X))\,)$.
        Indeed, if~$\mathscr{H}$ is a Hilbert space such that
        $\mathscr{B}\subseteq \mathscr{B}(\mathscr{H})$ (perhaps
        after renaming the elements of~$\mathscr{B}$),
        then~$\mathscr{B}'$
        is the closure (in the weak operator topology 
        on~$\mathscr{B}(\mathscr{H})$)
        of the smallest~$*$-subalgebra containing~$X$.
        Thus any element of~$\mathscr{B}'$
        is the limit of a filter --- a special type of net, 
        see paragraph~12 of~\cite{willard2004} ---
        of $*$-algebra terms over~$X$,
        of which there are no more than~$\#(\, \wp(\wp(X))\,)$.

	By a similar reasoning one sees
	that the opposite~$(\mathbf{W}^*_\mathrm{NCPsU})^\mathrm{op}$
	of the category
	of normal completely positive subunital linear maps
	between von Neumann algebras
	is Kleislian over~$(\WNMIU)^\mathrm{op}$.
	The existence of the adjoint
	to the inclusion~$\WNMIU\to \mathbf{W}^*_\mathrm{NCPsU}$
	is key
	in our
	construction
	of a model of Selinger and Valiron's quantum lambda calculus
	by von Neumann algebras, see~\cite{CW2016}.
\end{ex}

\subsection{Concrete description}
In this note we have 
shown that the embedding $U\colon \CMIU\longrightarrow \CPU$
has a left adjoint~$F$,
but we miss a concrete description of~$F\mathscr{A}$
for all but the simplest $C^*$-algebras~$\mathscr{A}$.
What constitutes a ``concrete description'' is perhaps 
a matter of taste or occasion,
but let us 
pose that it should at least enable us to describe 
the Eilenberg--Moore category~$\EM{FU}$ of the comonad~$FU$.
More concretely, it should settle the following problem.
\begin{prob}
        Writing 
        $\Cat{BOUS}$
        for the category of positive linear maps that preserve
        the unit between Banach order unit spaces,
        determine whether
        $\EM{FU} \cong \Cat{BOUS}$.

        (An \emph{order unit space} is an ordered vector space~$V$ over~$\R$
        with an element~$1$, the \emph{order unit},
        such that for all~$v\in V$ there is $\lambda\in[0,\infty)$
        such that~$-\lambda \cdot 1 \leq v \leq \lambda\cdot 1$.
        The smallest such~$\lambda$ is denoted by~$\|v\|$.
        See~\cite{kadison1951} for more details.
        If $v\mapsto \|v\|$ gives a complete norm,
        $V$ is called a \emph{Banach order unit space}.)
\end{prob}

\subsection{MIU versus PU}
A second ``problem'' is to give a physical description
(if there is any)
of what it means for a quantum program's semantics
to be a MIU-map (and not just a PU-map).
A step in this direction
might be to define for a $C^*$-algebra~$\mathscr{A}$,
a PU-map $\varphi\colon \mathscr{A}\ra \C$,
and $a,b\in \mathscr{A}$ the quantity
\begin{equation*}
        \mathrm{Cov}_\varphi(a,b)\ \eqdf\ \varphi(a^*b)
        \,-\,\varphi(a)^*\varphi(b)
\end{equation*}
and interpret it as the covariance between the observables~$a$ and~$b$
in state~$\varphi$ of the quantum system~$\mathscr{A}$.
Let $T\colon \mathscr{A}\longrightarrow \mathscr{B}$
be a PU-map between $C^*$-algebras
(so perhaps~$T$ is the semantics of a quantum program).
Then it is not hard to verify that~$T$ is a MIU-map
if and only if~$T$ preserves covariance, that is, 
\begin{equation*}
        \mathrm{Cov}_\varphi(\,Ta,\,Tb\,)
        \ = \ \mathrm{Cov}_{\varphi\circ T}(a,b)
	\qquad
	\text{for all~$a,b\in \mathscr{A}$}.
\end{equation*}
\section{Acknowledgements}
Example~\ref{ex:C2}
and Example~\ref{ex:C3}
were suggested by Robert Furber.
I'm grateful that Jianchao Wu and Sander Uijlen spotted several errors in a
previous version of this text.
Kenta Cho realised that the results
of this paper
might be used
to construct a model of the quantum lambda calculus.
I thank them, and
Bart Jacobs,
Sam Staton,
Wim Veldman, 
and Bas Westerbaan
for their help.

Funding was received from the
European Research Council under grant agreement \textnumero~320571.
\bibliographystyle{eptcs}
\bibliography{main}

\begin{thebibliography}{1}
\providecommand{\bibitemdeclare}[2]{}
\providecommand{\surnamestart}{}
\providecommand{\surnameend}{}
\providecommand{\urlprefix}{Available at }
\providecommand{\url}[1]{\texttt{#1}}
\providecommand{\href}[2]{\texttt{#2}}
\providecommand{\urlalt}[2]{\href{#1}{#2}}
\providecommand{\doi}[1]{doi:\urlalt{http://dx.doi.org/#1}{#1}}
\providecommand{\bibinfo}[2]{#2}

\bibitemdeclare{article}{CW2016}
\bibitem{CW2016}
\bibinfo{author}{Kenta \surnamestart Cho\surnameend} \&
  \bibinfo{author}{Abraham \surnamestart Westerbaan\surnameend}
  (\bibinfo{year}{2016}): \emph{\bibinfo{title}{Von Neumann Algebras form a
  Model for the Quantum Lambda Calculus}}.
\newblock {\sl
  \bibinfo{journal}{\href{https://arxiv.org/abs/1603.02133}{arXiv:1603.02133v1
  [cs.LO]}}}.

\bibitemdeclare{book}{devlin1993}
\bibitem{devlin1993}
\bibinfo{author}{Keith \surnamestart Devlin\surnameend} (\bibinfo{year}{1993}):
  \emph{\bibinfo{title}{The joy of sets: fundamentals of contemporary set
  theory}}.
\newblock \bibinfo{publisher}{Springer}, \doi{10.1007/978-1-4612-0903-4}.

\bibitemdeclare{incollection}{furber2013}
\bibitem{furber2013}
\bibinfo{author}{Robert \surnamestart Furber\surnameend} \&
  \bibinfo{author}{Bart \surnamestart Jacobs\surnameend}
  (\bibinfo{year}{2013}): \emph{\bibinfo{title}{From {{K}}leisli categories to
  commutative {{$C^*$}}-algebras: Probabilistic {{G}}elfand duality}}.
\newblock In: {\sl \bibinfo{booktitle}{Algebra and Coalgebra in Computer
  Science}}, \bibinfo{publisher}{Springer}, pp. \bibinfo{pages}{141--157},
  \doi{10.1007/978-3-642-40206-7\_12}.

\bibitemdeclare{book}{kadison1951}
\bibitem{kadison1951}
\bibinfo{author}{Richard~V. \surnamestart Kadison\surnameend}
  (\bibinfo{year}{1951}): \emph{\bibinfo{title}{A representation theory for
  commutative topological algebra}}.
\newblock \bibinfo{volume}{7}, \bibinfo{publisher}{American Mathematical
  Society}, \doi{10.1090/memo/0007}.

\bibitemdeclare{incollection}{lack2010}
\bibitem{lack2010}
\bibinfo{author}{Stephen \surnamestart Lack\surnameend} (\bibinfo{year}{2010}):
  \emph{\bibinfo{title}{A 2-categories companion}}.
\newblock In: {\sl \bibinfo{booktitle}{Towards higher categories}},
  \bibinfo{publisher}{Springer}, pp. \bibinfo{pages}{105--191},
  \doi{{10.1007/978-1-4419-1524-5\_4}}.

\bibitemdeclare{book}{maclane1998}
\bibitem{maclane1998}
\bibinfo{author}{Saunders \surnamestart Mac~Lane\surnameend}
  (\bibinfo{year}{1998}): \emph{\bibinfo{title}{Categories for the working
  mathematician}}.
\newblock \bibinfo{volume}{5}, \bibinfo{publisher}{springer},
  \doi{10.1007/978-1-4612-9839-7}.

\bibitemdeclare{article}{russo1966}
\bibitem{russo1966}
\bibinfo{author}{B.~\surnamestart Russo\surnameend} \& \bibinfo{author}{H.~A.
  \surnamestart Dye\surnameend} (\bibinfo{year}{1966}): \emph{\bibinfo{title}{A
  note on unitary operators in {{$C^*$-algebras}}}}.
\newblock {\sl \bibinfo{journal}{Duke Mathematical Journal}}
  \bibinfo{volume}{33}(\bibinfo{number}{2}), pp. \bibinfo{pages}{413--416},
  \doi{10.1215/S0012-7094-66-03346-1}.

\bibitemdeclare{article}{stinespring1955}
\bibitem{stinespring1955}
\bibinfo{author}{W.~Forrest \surnamestart Stinespring\surnameend}
  (\bibinfo{year}{1955}): \emph{\bibinfo{title}{Positive functions on
  {{$C^*$}}-algebras}}.
\newblock {\sl \bibinfo{journal}{Proceedings of the American Mathematical
  Society}} \bibinfo{volume}{6}(\bibinfo{number}{2}), pp.
  \bibinfo{pages}{211--216}, \doi{10.2307/2032342}.

\bibitemdeclare{book}{willard2004}
\bibitem{willard2004}
\bibinfo{author}{Stephen \surnamestart Willard\surnameend}
  (\bibinfo{year}{2004}): \emph{\bibinfo{title}{General topology}}.
\newblock \bibinfo{publisher}{Courier Dover Publications}.

\end{thebibliography}

\appendix
\section{Additional Proofs}
\begin{proof}[Proof of Lemma~\ref{lem:L}]
        Define $LC \eqdf FC$
        for all objects~$C$ of~$\Kl{UF}$
        and  
        \begin{equation*}
                Lf \ \eqdf\  \varepsilon_{FC_2}\circ Ff
        \end{equation*}
        for $f\colon C_1 \longrightarrow UFC_2$
        from~$\Cat{C}$.
        \shortlong{We leave it to the reader
        to check that
        this gives a functor~$L\colon \Kl{UF}\longrightarrow \Cat{D}$
        such that $U\circ L=G$ and~$L\circ V=F$;
        we will prove that~$L$ is unique as such.}{%
        We claim this gives a functor~$L\colon \Kl{UF}\longrightarrow \Cat{D}$.}

        \shortlong{}{ 
        \emph{($L$ preserves the identity)} \ 
        Let~$C$ be an object of~$\Kl{UF}$,
        that is, an object of~$\Cat{C}$.
        Then the identity on~$C$ in~$\Kl{UF}$ is~$\eta_{C}$.
        We have $L(\eta_C) = \varepsilon_{FC}\circ F\eta_C = \id{FC}$.

        \emph{($L$ preserves composition)} \ 
        Let $f\colon C_1\longrightarrow UFC_2$
        and $g\colon C_2\longrightarrow UFC_3$
        from~$\Cat{C}$ be given.
        We must prove that~$L(g\afterKl f) = Lg \circ Lf$.
        We have:
        \begin{alignat*}{3}
                L(g\afterKl f) 
                \ &= \ 
                L(\mu_{C_3} \circ UFg \circ f)
                &&\qquad\text{by def.~of~$g\afterKl f$}\\
                \ &=\ 
                \varepsilon_{FC_3} \circ F\mu_{C_3} \circ FUFg \circ Ff
                &&\qquad\text{by def.~of~$L$}\\
                \ &=\ 
                \varepsilon_{FC_3} \circ FU\varepsilon_{FC_3} 
                \circ FUFg \circ Ff
                &&\qquad\text{by def.~of~$\mu_{C_3}$}\\
                \ &=\ 
                \varepsilon_{FC_3} 
                \circ Fg \circ \varepsilon_{FC_2}\circ Ff
                &&\qquad\text{by nat.~of~$\eta$}\\
                \ &=\ 
                Lg \circ Lf
                &&\qquad\text{by def.~of~$L$}
        \end{alignat*}
        Hence~$L$ is a functor from~$\Kl{UF}$ to~$\Cat{D}$.

        Let us prove that $U\circ L =G$.
        For $f\colon C_1 \longrightarrow UFC_2$
        from~$\Cat{C}$
        we have
        \begin{alignat*}{3}
                ULf
                \ &=\ 
                U(\varepsilon_{FC_2}\circ Ff)
                \qquad&&
                \text{by def.~of~$L$}\\
                \ &=\ 
                U\varepsilon_{FC_2}\circ UFf
                \qquad&&
                \text{}\\
                \ &=\ 
                \mu_{C_2}\circ UFf
                \qquad&&
                \text{by def.~of~$\mu_{C_2}$}\\
                \ &=\ 
                Gf
                \qquad&&
                \text{by def.~of~$Gf$}.
        \end{alignat*}

        Let us prove that $L\circ V = F$.
        For $f\colon C_1\longrightarrow C_2$
        from~$\Cat{C}$ be given,
        we have
        \begin{alignat*}{3}
                LVf
                \ &=\ 
                L(\eta_{C_2}\circ f)
                \qquad&&
                \text{by def.~of~$V$}\\
                \ &=\ 
                \varepsilon_{FC_2} \circ F \eta_{C_2}\circ Ff
                \qquad&&
                \text{by def.~of~$L$}\\
                \ &=\ 
                Ff
                \qquad&&
                \text{by counit--unit eq.}
        \end{alignat*}

        We have proven that there is a functor~$L\colon \Kl{UF}\ra \Cat{D}$
        such that $U\circ L=G$ and $L\circ V = F$.
        We must still prove that it is as such unique.%
        }
        
        Let~$L'\colon \Kl{UF} \ra \Cat{D}$
        be a functor such that~$U\circ L'=G$
        and $L'\circ V = F$.
        We must show that~$L=L'$.
        Let us first prove that~$L'$ and~$L$ agree on objects.
        Let $C$  be an object of~$\Kl{UF}$,
        i.e., $C$ is an object of~$\Cat{C}$.
        Since $L'\circ V = F$
        and $VC=C$ we have $L'C=L'VC = FC = LC$.
        Now, let $f\colon C_1 \rightarrow UFC_2$
        from~$\Cat{C}$ be given
        (so $f$ is a morphism in~$\Kl{UF}$
        from~$C_1$ to~$C_2$).
        We must show that~$L'f = LU\equiv \varepsilon_{FC_2}\circ Ff$.
        Note that since~$F$ is the left adjoint of~$U$
        there is a unique 
        morphism~$\overline{f}\colon FC_1\longrightarrow FC_2$
        in~$\Cat{D}$
        such that~$U\overline{f}\circ \eta_{C_1}=f$.
        To prove that $L'f=Lf$,
        we show that both~$Lf$ and~$L'f$ have this property.
        We have
        \begin{alignat*}{3}
                UL'f \circ \eta_{C_1}
                \ &=\ 
                Gf \circ \eta_{C_1} &&
                \text{as $U\circ L'=G$ by assump.}\\
                \ &=\ 
                \mu_{C_2}\circ UF f \circ \eta_{C_1} \qquad&&
                \text{by def.~of~$G$}\\
                \ &=\ 
                \mu_{C_2}\circ \eta_{UFC_2} \circ f  &&
                \text{by nat.~of~$\eta$}\\
                \ &=\ 
                f  &&
                \text{as~$UF$ is a monad.}
        \end{alignat*}
        By a similar argument
        we get $ULf\circ \eta_{C_1} = f$.
        Hence $Lf=L'f$.
\end{proof}

\begin{proof}[Proof of Theorem~\ref{thm:kleisli-beck}]
        We use the symbols from Notation~\ref{nt:beck}.
        
        \ref{thm:kleisli-beck-i}$\Longrightarrow$
        \ref{thm:kleisli-beck-ii}\ 
        Suppose that~$L$ is an isomorphism.
        We must prove that~$F$ is bijective on objects.
        Note that~$F=L\circ V$,
        so it suffices to show that both~$L$ and~$V$
        are bijective on objects.
        Clearly,~$L$ is bijective on objects
        as~$L$ is an isomorphism,
        and~$V\colon \Cat{C} \longrightarrow \Kl{UF}$
        is bijective on objects
        since the objects of~$\Kl{UF}$ are those of~$\Cat{C}$
        and $VC=C$ for all~$C$ from~$\Cat{C}$.

        \ref{thm:kleisli-beck-ii}$\Longrightarrow$
        \ref{thm:kleisli-beck-i}\ 
        Suppose that~\ref{thm:kleisli-beck-ii} holds.
        We prove that~$L$ is an isomorphism
        by giving its inverse.
        Let~$D$ be an object from~$\Cat{D}$.
        Note that since~$F$ is bijective on objects
        there is a unique object~$C$ from~$\Cat{C}$ such 
        that~$FD=C$. Define $KC \eqdf D$.

        Let $g\colon D_1 \ra D_2$ from~$\Cat{D}$
        be given.  Note that
        by definition of~$K$ we have:
        \begin{equation*}
                \xymatrix{
                        KD_1
                        \ar[rr]^{\eta_{KD_1}}
                        &&
                        UFKD_1
                        \ar@{=}[r]
                        &
                        UD_1
                        \ar[rr]^{Ug}
                        &&
                        UD_2
                        \ar@{=}[r]
                        &
                        UFKD_2
                }
        \end{equation*}
        Now, define
        $Kg\colon KD_1 \ra UFKD_2$ in~$\Cat{D}$ by
        $Kg \ \eqdf\ Ug\circ \eta_{KD_1}$.

        \shortlong{We leave it to the reader
                to check that this gives
                a functor $K\colon \Cat{D}\longrightarrow \Kl{UF}$;
        we will show that~$K$ is the inverse of~$L$.}{
        We claim that this 
        gives a functor $K\colon \Cat{D}\longrightarrow \Kl{UF}$.

        \emph{($K$ preserves the identity)}\ 
        For an object~$D$ of~$\Cat{D}$
        we have 
        \begin{equation*}
                K\id{D} \ =\  U\id{D}\circ  \eta_{KD} \ =\  \eta_{KD},
        \end{equation*}
        and~$\eta_{KD}$ is the identity on~$KD$ in~$\Kl{UF}$.

        \emph{($K$ preserves composition)}\ 
        Let $f\colon D_1 \longrightarrow D_2$
        and $g\colon D_2 \longrightarrow D_3$
        from~$\Cat{D}$ be given.
        We must prove that~$K(g\circ f) = K(g)\afterKl K(f)$.
        We have
        \begin{alignat*}{3}
                K(g)\afterKl K(f)
                \ &=\ 
                \mu_{KD_3} \circ UFKg \circ Kf
                &&\qquad
                \text{by def.~of~$\afterKl$}\\
                \ &=\ 
                \mu_{KD_3} \circ UFUg\circ UF\eta_{KD_2} \circ Uf
                \circ \eta_{KD_1}
                &&\qquad
                \text{by def.~of~$K$}\\
                \ &=\ 
                U\varepsilon_{D_3} \circ UFUg\circ UF\eta_{KD_2} \circ Uf
                \circ \eta_{KD_1}
                &&\qquad
                \text{by def.~of~$\mu$}\\
                \ &=\ 
                Ug\circ U\varepsilon_{D_2} \circ UF\eta_{KD_2} \circ Uf
                \circ \eta_{KD_1}
                &&\qquad
                \text{by nat.~of~$\varepsilon$}\\
                \ &=\ 
                Ug \circ Uf
                \circ \eta_{KD_1}
                &&\qquad
                \text{by counit--unit eq.}\\
                \ &=\ 
                K(g \circ f)
                &&\qquad
                \text{by def of~$K$.}
        \end{alignat*}

        Hence~$K$ is a functor from~$\Cat{D}$ to~$\Kl{UF}$.
        We will show that~$K$ is the inverse of~$L$.%
        }
        For this we must prove that $K\circ L = \id{\Cat{D}}$
        and $L\circ K = \id{\Kl{UF}}$.

        For a morphism $g\colon D_1 \longrightarrow D_2$
        from~$\Cat{D}$, we have
        \begin{alignat*}{3}
                LKg
                \ &=\ 
                L(Ug\circ \eta_{KD_1}) 
                &&\qquad\text{by def.~of~$K$}\\
                \ &=\ 
                \varepsilon_{FKD_2}\circ FUg\circ F\eta_{KD_1} 
                &&\qquad\text{by def.~of~$L$}\\
                \ &=\ 
                g\circ \varepsilon_{FKD_1}\circ F\eta_{KD_1} 
                &&\qquad\text{by nat.~of~$\varepsilon$}\\
                \ &=\ 
                g
                &&\qquad\text{by counit--unit eq.}
        \end{alignat*}

        For a morphism $f\colon C_1 \longrightarrow UFC_2$
        in~$\Cat{C}$ we have
        \begin{alignat*}{3}
                KLf
                \ &=\ 
                K(\varepsilon_{FC_2}\circ Ff)
                &&\qquad\text{by def.~of~$L$}\\
                KLfdd
                \ &=\ 
                U\varepsilon_{FC_2}\circ UFf \circ  \eta_{KFC_1}
                &&\qquad\text{by def.~of~$K$}\\
                \ &=\ 
                U\varepsilon_{FC_2}\circ \eta_{UFC_2}\circ f 
                &&\qquad\text{by nat.~of~$\eta$}\\
                \ &=\ 
                f 
                &&\qquad\text{by counit--unit~eq.}
        \end{alignat*}
        Hence~$K$ is the inverse of~$L$,
        so~$L$ is an isomorphism.
\end{proof}
 }
\end{document}